\documentclass[12pt]{amsart}%
\usepackage{verbatim}
\usepackage{latexsym}
\usepackage{amsmath,enumitem}
\usepackage{amsthm}
\usepackage{amsfonts}
\usepackage{mathtools}
\usepackage{amssymb}
\usepackage[table]{xcolor}
\usepackage{xfrac}
\usepackage{hyperref}
\usepackage{graphicx}%
\usepackage{tikz,tikz-cd}
\graphicspath{ {./} }

\newcommand{\C}{\mathbb{C}}

\newtheorem{thm}{Theorem}

\newtheorem{prop}[thm]{Proposition}
\newtheorem{coro}[thm]{Corollary}
\newtheorem{lem}[thm]{Lemma}

\theoremstyle{remark}
\newtheorem*{remark}{Remark}

\newcommand{\m}{\mathrm{m}}
\newcommand{\ol}[1]{\overline{#1}}

\newcommand{\logplus}{\log^+}

\renewcommand{\pmod}[1]{{\ifmmode\text{\rm\ (mod~$#1$)}\else\discretionary{}{}{\hbox{ }}\rm(mod~$#1$)\fi}}

\author{Matilde Lal\'in}
\author{Siva Sankar Nair}

\address{Matilde Lal\'in:  D\'epartement de math\'ematiques et de statistique,
                                    Universit\'e de Montr\'eal,
                                    CP 6128, succ. Centre-ville,
                                     Montreal, QC H3C 3J7, Canada}\email{matilde.lalin@umontreal.ca}

\address{Siva Sankar Nair:  D\'epartement de math\'ematiques et de statistique,
                                    Universit\'e de Montr\'eal,
                                    CP 6128, succ. Centre-ville,
                                     Montreal, QC H3C 3J7, Canada}\email{siva.sankar.nair@umontreal.ca}
\begin{document}

\title{An invariant property of Mahler measure}

\begin{abstract}
 We exhibit a change of variables that maintains the Mahler measure of a given polynomial. This method leads to the construction of 
highly non-trivial polynomials with given Mahler measure and settles some conjectural numerical  formulas due to Boyd and Brunault.
\end{abstract}


\maketitle

\section{Introduction}
Given a non-zero rational function $P \in \C(x_1,\dots,x_n)$, its (logarithmic) Mahler measure is given by 
\begin{equation*}
 \m(P):=\frac{1}{(2\pi i)^n}\int_{\mathbb{T}^n}\log|P(x_1,\dots, x_n)|\frac{\mathrm{d}  x_1}{x_1}\cdots \frac{\mathrm{d}  x_n}{x_n},
\end{equation*}
where the integration is taken over the unit torus $\mathbb{T}^n=\{(x_1,\dots,x_n)\in \mathbb{C}^n : |x_1|=\cdots=|x_n|=1\}$ with respect to the Haar measure. 

This quantity first appeared (for one variable-polynomials) in considerations by Lehmer \cite{Le} in his work related to the generalization of Mersenne numbers, and was later extended to polynomials with several variables by Mahler  \cite{Mah} in applications to polynomial heights. About 20 years later Smyth \cite{S1,B1} proved the  formulas
\[\m(x+y+1) = \frac{3 \sqrt{3}}{4 \pi} L(\chi_{-3},2),\]
\begin{equation}\label{eq:Smyth3}
\m(x+y+z+1)=\frac{7}{2\pi^2} \zeta(3),
\end{equation}
where  $L(\chi_{-3},s)$ is the Dirichlet $L$-function in the character of conductor 3 and $\zeta(s)$ is the Riemann zeta function.

These formulas opened the door to a wave of research in Mahler measure and special values of functions with arithmetical significance. Consider, for example, the following formula conjectured by Deninger  \cite{Deninger} and Boyd \cite{Bo98}  and proven by Rogers and Zudilin \cite{RZ14}:
\[\m\left(x+\frac{1}{x}+y+\frac{1}{y}+1\right)=L'(E_{15a8},0),\]
where $L(E_{15a8},s)$ denotes the $L$-function associated to the elliptic curve $15a8$. 

The appearance of values of $L$-functions in certain Mahler measure formulas was explained by Deninger in terms of Beilinson's conjectures via relationships with regulators. (See also the works of Boyd \cite{Bo98} and Rodriguez-Villegas \cite{RV} for additional insights, and the book of Brunault and Zudilin \cite{BrunaultZudilin} for more details.) Likewise, the cases involving the Riemann zeta function and Dirichlet $L$-functions have been linked to particular applications of the Borel regulator and evaluations of polylogarithms in algebraic numbers \cite{Boyd-RV,Boyd-RV2,L-Duke,L08}. 

A particularly interesting formula was proven by Condon \cite{Condon}:
\begin{equation}\label{eq:Condon}
\m(x+1+(x-1)(y+z))=\frac{28}{5\pi^2}\zeta(3).
\end{equation}
Although this identity seems similar to \eqref{eq:Smyth3}, it is indeed much harder to prove. While the right-hand side of \eqref{eq:Smyth3} is the result of evaluating the trilogarithm in $1$ and $-1$, the right-hand side of \eqref{eq:Condon} is the result of evaluating the trilogarithm in combinations of $\varphi=\frac{1+\sqrt{5}}{2}$. This complicates the computation considerably, independently of whether one uses elementary methods as in \cite{Condon} or the regulator as in \cite{L-Duke}.

In \cite{Boyd06}, Boyd proposed the numerical study of polynomials of the form $a(x)+b(x)y+c(x)z$, where $a(x), b(x), c(x)$ are products of cyclotomic polynomials. The focus on this particular class of polynomials was motivated by the Cassaigne--Maillot formula for the Mahler measure of $a+by+cz$ \cite{CM}, which has an expression that is specially well-suited for numerical integration. The exploration of such polynomials led to the discovery of several interesting numerical formulas involving $L'(E,-1)$ for various elliptic curves \cite{Lalin-Crelle}. For example, Boyd discovered
\begin{equation}\label{eq:L21}
\m(1+(x-1)y+(x+1)z)\stackrel{?}{=}\frac{5}{4}L'(E_{21a1},-1),
\end{equation}
where the question mark denotes a numerical identity that has been verified to at least 20 decimal places, and also  
\begin{equation}
\m(x^2+x+1+(x^2-1)(y+z))\stackrel{?}{=}\frac{28}{5\pi^2}\zeta(3),\label{eq:2}
\end{equation}
which involves, again, the term $\frac{28}{5\pi^2}\zeta(3)$.

More recently Brunault further pursued these computations (with higher degree cyclotomic polynomials) and discovered various numerical formulas yielding $\frac{28}{5\pi^2}\zeta(3)$ such as
\begin{align}
\m(x^4-x^3+x^2-x+1+ (x^4-x^3+x-1)(y+z))&\stackrel{?}{=}\frac{28}{5\pi^2}\zeta(3),\label{eq:4}\\
\m(x^5+x^4+x+1+ (x^5-1)(y+z))&\stackrel{?}{=}\frac{28}{5\pi^2}\zeta(3),\label{eq:5}
\end{align}
and several others. 

We started this project by investigating whether the above polynomials having Mahler measure $\frac{28}{5\pi^2}\zeta(3)$ had anything in common with each other. We discovered that they all could be obtained from Condon's polynomial in \eqref{eq:Condon} by certain changes of variables. In fact, such changes of variables provide a method for generating arbitrarily many rational functions that have the same Mahler measure, giving rise to highly non-trivial identities. Moreover, we can also obtain families of conjectures by applying the change of variables on polynomials such as the one in \eqref{eq:L21}.

Before stating our main result, we establish some notation. For a polynomial $g(x) \in \C[x]$, such that $g(x)=\sum_{j=0}^d g_j x^j$ with $g_d\not = 0$, denote by \[\ol g(x)=\sum_{j=0}^d \ol g_j x^j,\] the polynomial resulting from conjugating the coefficients of $g(x)$. 

We prove the following result. 
\begin{thm}\label{thm:main}
Let $P(x,y_1,\dots,y_n)$ be a polynomial over $\C$ in the variables $x,y_1,\dots,y_n$.  Let $g(x)\in\C[x]$ be a polynomial in the variable $x$ with all roots outside the unit disc, and for an integer $k$ greater than the degree of $g(x)$ and a complex number $\lambda$ of absolute value one, let $f(x)=\lambda x^k\ol g(x^{-1})$.  We denote by $\widetilde{P}$ the rational function obtained by replacing $x$ by $f(x)/g(x)$ in $P$. Then
\[
\m(P)=\m(\widetilde{P}).
\]
\end{thm}

Using this theorem, equations \eqref{eq:2}, \eqref{eq:4},  and \eqref{eq:5} follow by taking $P(x,y,z)=x+1+(x-1)(y+z)$, which is the polynomial from Condon's result in \eqref{eq:Condon},  with $g(x)=x+2, x^2-2x+2,$ and $x^4+x+2$ along with $k=2, 4,$ and $5$ respectively, and $\lambda =1$. We have a term of $\m(g)=\log2$ in all these cases, which cancels with the contribution from a factor of $2$ occurring in $\widetilde P$, leading  us to the desired identities. 

As another example we can prove the following formulas 
\begin{align}
\m((x^2+x+1)(1+w)(1+u)z+(x^2-1)(1-w)(1+y)) =&  \frac{93}{\pi^4} \zeta(5), \label{eq:93zeta5}\\
\m((x^4-x^3+x^2-x+1)(1+w)(1+u)z+ (x^4-x^3+x-1)(1-w)(1+y))=& \frac{93}{\pi^4} \zeta(5), \label{eq:93zeta5a}\\
\m((x^5+x^4+x+1)(1+w)(1+u)z+ (x^5-1)(1-w)(1+y))=& \frac{93}{\pi^4} \zeta(5),\label{eq:93zeta5b}
\end{align}
by applying Theorem \ref{thm:main} to a result proven in \cite{nvar}.

In addition, by applying this method to \eqref{eq:L21}, we obtain a family of conjectures involving $L'(E_{21a1},-1)$, including
\begin{align}
\m\left(\frac{x+2}{2}+ (x^2+x+1)y+(x^2-1)z\right)&\stackrel{?}{=}\frac{5}{4}L'(E_{21a1},-1),\label{eq:L21-2}\\
\m\left(\frac{x^2-2x+2}{2}+(x^4-x^3+x^2-x+1)y+ (x^4-x^3+x-1)z\right)&\stackrel{?}{=}\frac{5}{4}L'(E_{21a1},-1), \label{eq:L21-3}\\
\m\left(\frac{x^4+x+2}{2}+(x^5+x^4+x+1)y+ (x^5-1)z\right)&\stackrel{?}{=}\frac{5}{4}L'(E_{21a1},-1). \label{eq:L21-4}
\end{align}

This article is organized as follows. In Section \ref{sec:intermediate} we consider some results needed in preparation for the proof of Theorem \ref{thm:main}, which is included in Section \ref{sec:proof}. In Section \ref{sec:applications} we discuss some applications of Theorem \ref{thm:main}, mainly to families of rational functions with an arbitrary number of variables. Finally, Section \ref{sec:genus} sketches the relation between Mahler measures and regulators, with the goal of giving more context to some of the new formulas proven in this article that involve polynomials defining curves of high genus. 

\section*{Acknowledgements} The authors would like to thank David Boyd, Fran\c cois Brunault, Andrew Granville, and Wadim Zudilin for helpful discussions, and also gratefully acknowledge the support of Sujatha Ramdorai. The authors are thankful to the referee for their helpful remarks. This work is supported by  the Natural Sciences and Engineering Research Council of Canada, Discovery Grant 355412-2022, the Fonds de recherche du Qu\'ebec - Nature et technologies, Projet de recherche en \'equipe 300951,  the Institut des sciences math\'ematiques and the Centre de recherches math\'ematiques.

\section{Some intermediate results} \label{sec:intermediate}

Theorem \ref{thm:main} depends on the nature of the change of variables $x \rightarrow f(x)/g(x)$ and how this affects the Mahler measure. In this section, we address this question and prove some intermediate results regarding the Mahler measure of certain univariate polynomials. These results will play an important role in the proof of the main result later on.

We first present the following lemma concerning the roots of a polynomial obtained from the polynomials $f(x)$ and $g(x)$ that appear in the change of variables.

\begin{lem}{\cite[Lemma 1]{Lax}, \cite[Theorem 1]{Chen}}\label{lem:roots}
Let $g(x)\in\C[x]$ be such that all its roots have absolute value greater than or equal to one,  let $k$ be an integer such that $k\ge\mathrm{deg }\;(g)$ and let $f(x)=\lambda x^k \ol g(x^{-1})$, where $\lambda$ is a complex number with absolute value one.  For any complex number $\beta$, consider 
\[
{\Gamma_\beta}(x)=f(x)+\beta g(x) \in \C[x]. 
\]
Then we have the following:
\begin{enumerate}[label=(\alph*)]
\item If $\vert\beta\vert< 1$, then all roots of ${\Gamma_\beta}(x)$ have absolute value less than  one. 
\item If $\vert\beta\vert> 1$, then all roots of ${\Gamma_\beta}(x)$ have absolute value greater than  one. 
\item If $\vert\beta\vert=1$, then all roots of ${\Gamma_\beta}(x)$ have absolute value one.
\end{enumerate}
\end{lem}

\begin{remark}
 This result has a rich history, as it was reproven several times.
 As stated, the original result is due to Lax \cite{Lax} (for the case $k=\deg(g)$) and  Chen \cite{Chen} (for the general case). Both consider the case when $|\beta|=1$ in their statements, but their proofs include the three cases stated in Lemma \ref{lem:roots}. This result was later rediscovered by Lal\'in and Smyth \cite[Theorem 1]{LalinSmyth} (stated for the case $k>\deg(g)$). 
A converse was stated and proven by Cohn \cite{Cohn}. Further results investigate the interlacing of the roots of $\Gamma_\beta$ as $\beta$ varies in the unit circle, see \cite{Lalin-Smyth-addendum} and the references therein for more information. 
\end{remark}

The above result is true for any $k$ greater than or equal to $d$, the degree of $g(x)$. However, for the rest of our results to hold true, we would require $k$ to be strictly greater than $d$.

The above lemma gives us the following corollary using Jensen's formula.

\begin{coro}\label{coro:base}
Let $f$ and $g$ be as before with $k$ strictly greater than the degree of $g(x)$. Then for any $\beta\in\C$, we have
\[
\m(\Gamma_\beta(x))=\log|g_0|+\logplus|\beta|,
\]
where $\logplus|x|:=\log\max\{1,|x|\}$.
\end{coro}
\begin{proof}
Note that $\Gamma_\beta(x)=\lambda x^k \ol g(x^{-1})+\beta g(x)$, and since $k>\mathrm{deg}\;g$, the leading coefficient here is $\lambda \ol{g_0}$ and the constant term is $\beta g_0$.  By Lemma \ref{lem:roots}, we know that if $\alpha_j$ are the roots of $\Gamma_\beta$, then either $|\alpha_j|\ge1$ for all $j$, or $|\alpha_j|\le 1$ for all $j$. In both cases, this means that by Jensen's formula, 
\begin{align*}
\m(\Gamma_\beta)-\log|\lambda \ol{g_0}|=\sum_j \logplus|\alpha_j|&=\logplus\Big|\prod_j \alpha_j \Big|\\
&=\logplus\left|\frac{\beta g_0}{\lambda \ol{g_0}}\right|=\logplus|\beta|,
\end{align*}
and since $\log|\lambda \ol{g_0}|=\log|g_0|$, we get the desired result.
\end{proof}
The next statement is essentially a restatement of Corollary \ref{coro:base} in slightly more general terms. It will be used to prove Theorem \ref{thm:main}  in the next section.
\begin{prop}\label{prop:mainprop}
Let $f$ and $g$ be as before with $k$ strictly greater than the degree of $g(x)$. Then for any $\alpha,\beta\in\C$,  not both zero, we have
\[
\m(\alpha f +\beta g) =\m(g)+ \log\max \{ |\alpha|, |\beta|\}.
\]
\end{prop}
\begin{remark} Note that since $g$ has all its roots outside the unit circle, we have $\m(g)=\log|g_0|$.  Similarly, since $f$ has all its roots inside the unit circle, its Mahler measure is given by its leading coefficient, which is $\lambda \ol{g_0}$. Thus, $\m(f)=\m(g)=\log|g_0|.$ 
\end{remark}

\begin{proof}
Consider 
\[\Gamma(x)=\alpha f(x) +\beta g(x)=\lambda \alpha x^k \overline{g}(x^{-1})+\beta g(x).\]
Then $\deg(\Gamma)=k$ as before.  We first consider the case $\alpha\neq0$. Without loss of generality, we can then assume that $\alpha=1$.  Using Corollary \ref{coro:base} with $\Gamma_\beta(x)=\Gamma(x)$, we have
\begin{align*}
\m(\Gamma(x))&=\log|g_0|+\logplus|\beta|\\
		       &=\m(g)+\log\max\{1,|\beta|\}\\
		       &=\m(g)+ \log\max \{ |\alpha|, |\beta|\},
\end{align*}
as needed. Coming to the second case where $\alpha=0$, then
\[
\m(\alpha f +\beta g)=\m(\beta g)=\m(g)+\log|\beta|,
\]
and since $\log\max\{|\alpha|,|\beta|\}=\log|\beta|$, we get the desired result.
\end{proof}

Before moving on to the proof of the main theorem, we will also briefly discuss the existence of the Mahler measure integral for any polynomial in general. Although a fundamental and well-known fact, we prove the result here for completion, due to its central role in the proof of Theorem \ref{thm:main}. 

Let $F(y_1,y_2,\dots,y_{n+1})$ be a polynomial in $\C[y_1,\dots,y_{n+1}]$ in $(n+1)$ variables. We may write $F$ as a polynomial in $y_{n+1}$ with coefficients being $n$-variable polynomials $c_i\in\C[y_1,\dots,y_{n}]$,  as 
\[
F=c_0+c_1y_{n+1}+\cdots+c_ky_{n+1}^k,
\]
for some non-negative integer $k$. We may also factorize $F$ as
\[
F=c_k\prod_{j=1}^k\Big(y_{n+1}-\Delta_j(y_1,y_2,\dots,y_{n})\Big),
\]
where $\Delta_j(y_1,y_2,\dots,y_{n})$ are algebraic functions in $y_1,\dots,y_n$ with appropriately chosen branch cuts. Now consider the following lemma:
\begin{lem}{\cite[Lemma 3.7]{EverestWard}}\label{lem:exist}
Let $F$ be a polynomial as above. Then the Mahler measure $\m(F)$ of $F$  exists and we can write 
\[
\m(F)=\m(c_k)+\frac{1}{(2\pi i)^{n}}\sum_{j=1}^k\int_{|y_1|=1}\cdots \int_{|y_n|=1}\logplus|\Delta_j(y_1,\dots,y_n)|\;\frac{\mathrm{d} y_n}{y_n}\cdots\frac{\mathrm{d} y_1}{y_1},
\]
where all integrals involved converge.
\end{lem}
\begin{proof}
We proceed by induction on the number of variables in the polynomial $F$. In the base case where $n=1$, the $\Delta_j$'s are simply the roots of $F$, and the result follows by Jensen's formula. Now assume the result is true for $n$ variables. Using the factorization of $F$, we can write
\begin{align}
\m(F)&=\frac{1}{(2\pi i)^{n+1}}\int_{|y_1|=1}\cdots \int_{|y_n|=1}\int_{|y_{n+1}|=1} \log|F|\;\frac{\mathrm{d} y_{n+1}}{y_{n+1}}\;\frac{\mathrm{d} y_n}{y_n}\cdots\frac{\mathrm{d} y_1}{y_1}\nonumber\\
&=\m(c_k)+\frac{1}{(2\pi i)^{n+1}}\sum_{j=1}^k\int_{|y_1|=1}\cdots \int_{|y_n|=1}\underbrace{\Big(\int_{|y_{n+1}|=1} \log\Big|y_{n+1}-\Delta_j\Big|\;\frac{\mathrm{d} y_{n+1}}{y_{n+1}}\Big)}_I\;\frac{\mathrm{d} y_n}{y_n}\cdots\frac{\mathrm{d} y_1}{y_1}.\nonumber
\end{align}
Recall that if $\beta \in\C$, then Jensen's formula implies that
\[
\m( x+\beta)=\frac{1}{2\pi i}\int_{|x|=1}\log|x+\beta|=\logplus|\beta|.
\]
Thus, for constant $y_1,y_2,\dots, y_n$, the inner integral above is given by $I=2\pi i\logplus|\Delta_j(y_1,\dots,y_n)|$, and we have 
\begin{align}\label{eq:integral}
\m(F)=\m(c_k)+\frac{1}{(2\pi i)^{n}}\sum_{j=1}^k\int_{|y_1|=1}\cdots \int_{|y_n|=1}\logplus|\Delta_j(y_1,\dots,y_n)|\;\frac{\mathrm{d} y_n}{y_n}\cdots\frac{\mathrm{d} y_1}{y_1}.
\end{align}
The polynomial $c_k$ has $n$ variables and by the induction hypothesis, it follows that $\m(c_k)$ exists.  It remains to show that the integral in \eqref{eq:integral} also exists.  First, we note that the torus $\mathbb{T}^{n+1}$ defined by $\{(y_1,y_2,\dots,y_{n+1})\in\C^{n+1}\;:\;|y_j|=1\}$ is compact, and hence the polynomial $F$ must be bounded above on this region, and therefore $\m(F)$ is bounded. Since $\m(c_k)$ is finite, this means the integral in \eqref{eq:integral} is also bounded from above.  Secondly, if we write $y_j=e^{i\theta_j}$ for a real variable $\theta_j$, then 
\[
\frac{\mathrm{d} y_j}{y_j}=i\;\mathrm{d} \theta_j,
\]
for each $j$. Taking into account the factor of $(i)^n$ in the denominator outside the summation, we observe that the integrand is $\logplus|\Delta_j|$, which is totally real and non-negative. In addition,  although the algebraic functions $\Delta_j$ may not be continuous individually,  the multiset of values $\{\Delta_j(y_1,\dots,y_n)\}$ depends continuously on $(y_1,\dots,y_n)$, away from the poles of the $\Delta_j$. Thus, the integrands are non-negative and continuous, and along with the fact that the integrals are bounded above, this shows that they all converge. This completes the induction and we conclude that $\m(F)$ exists and that the integral in \eqref{eq:integral} converges.
\end{proof}

\section{Proof of Theorem \ref{thm:main}} \label{sec:proof}
We are now ready to prove our main result.  Recall that we replace the variable $x$ by $f(x)/g(x)$ in the polynomial $P$ to obtain the rational function $\widetilde P$. We will show that the Mahler measure computation for both $P$ and $\widetilde P$ is given by the same integral.

\begin{proof}[Proof of Theorem \ref{thm:main}]
We first write $P$ as a polynomial in $x$ with coefficients in $\C[y_1,\dots,y_n]$ as
\[
P=P_0+P_1x+\cdots+P_\ell x^\ell,
\]
where $P_j\in\C[y_1,\dots,y_n]$ for each $1\le j\le \ell$.  We can now factorize this as
\[
P=P_\ell\prod_{j=1}^\ell\Big(x-\Delta_j(y_1,\dots,y_n)\Big),
\] 
where the $\Delta_j$'s are algebraic functions in $y_1,\dots,y_n$ with appropriately chosen branch cuts.  Then, using Lemma \ref{lem:exist}, the Mahler measure of $P$ is given by

\begin{align}\label{eq:leftterm}
\m(P)=\m(P_\ell)+\frac{1}{(2\pi i)^{n}}\sum_{j=1}^\ell\int_{|y_1|=1}\cdots \int_{|y_n|=1}\logplus|\Delta_j(y_1,\dots,y_n)|\;\frac{\mathrm{d} y_n}{y_n}\cdots\frac{\mathrm{d} y_1}{y_1}.
\end{align}

We now move on to the evaluation of $\m(\widetilde{P})$, the rational function obtained from $P$ by replacing $x$ by $f(x)/g(x)$, which gives
\begin{align*}
\widetilde{P}&=P_\ell\prod_{j=1}^\ell\left(\frac{f(x)}{g(x)}-\Delta_j(y_1,\dots,y_n)\right)\\
 &=(g(x))^{-\ell}P_\ell\prod_{j=1}^\ell\Big({f(x)}-g(x)\Delta_j(y_1,\dots,y_n)\Big).
\end{align*}
Thus, the  Mahler measure of $\widetilde{P}$ is given by
\begin{align*}
\m(\widetilde{P})&=\frac{1}{(2\pi i)^{n+1}}\int_{|y_1|=1}\cdots \int_{|y_n|=1}\int_{|x|=1} \log|\widetilde{P}|\;\frac{\mathrm{d} x}x\;\frac{\mathrm{d} y_n}{y_n}\cdots\frac{\mathrm{d} y_1}{y_1}\nonumber\\
&=\m(P_\ell)-\ell\cdot \m(g)+\frac{1}{(2\pi i)^{n+1}}\sum_{j=1}^\ell\int_{|y_1|=1}\cdots \int_{|y_n|=1}\underbrace{\Big(\int_{|x|=1} \log\Big|{f(x)}-g(x)\Delta_j\Big|\;\frac{\mathrm{d} x}x\Big)}_J\;\frac{\mathrm{d} y_n}{y_n}\cdots\frac{\mathrm{d} y_1}{y_1}.\nonumber
\end{align*}
Keeping the $y_1,y_2,\dots,y_n$ constant and using Proposition \ref{prop:mainprop}, we can write the inner integral $J$ as
\[
J=\int_{|x|=1} \log\Big|{f(x)}-g(x)\Delta_j\Big|\;\frac{\mathrm{d} x}x=2\pi i\,\m(f-g\Delta_j)=2\pi i\Big(\m(g)+\logplus{\Delta_j}\Big),
\]
for each $1\le j\le \ell$. Noting that
\begin{align*}
\sum_{j=1}^\ell \int_{|y_1|=1}\cdots \int_{|y_n|=1}\Big(2\pi i\,\m(g)\Big)\;\frac{\mathrm{d} y_n}{y_n}\cdots\frac{\mathrm{d} y_1}{y_1}
&=2\pi i\,\m(g)\sum_{j=1}^\ell \int_{|y_1|=1}\cdots \int_{|y_n|=1}1\;\frac{\mathrm{d} y_n}{y_n}\cdots\frac{\mathrm{d} y_1}{y_1}\\
&=2\pi i\,\m(g)\sum_{j=1}^\ell (2\pi i)^n=\ell\cdot \m(g)\cdot (2\pi i)^{n+1},
\end{align*}
we can rewrite the Mahler measure of $\widetilde{P}$ as
\begin{align}\label{eq:rightterm}
\m(\widetilde{P})&=\m(P_\ell)-\ell \cdot \m(g)+\ell \cdot \m(g)+\frac{1}{(2\pi i)^{n}}\sum_{j=1}^\ell \int_{|y_1|=1}\cdots \int_{|y_n|=1}\logplus|\Delta_j(y_1,\dots,y_n)|\;\frac{\mathrm{d} y_n}{y_n}\cdots\frac{\mathrm{d} y_1}{y_1}\nonumber\\
&=\m(P_\ell)+\frac{1}{(2\pi i)^{n}}\sum_{j=1}^\ell \int_{|y_1|=1}\cdots \int_{|y_n|=1}\logplus|\Delta_j(y_1,\dots,y_n)|\;\frac{\mathrm{d} y_n}{y_n}\cdots\frac{\mathrm{d} y_1}{y_1}.
\end{align}
We observe that equations \eqref{eq:leftterm} and \eqref{eq:rightterm} evaluate to the same expression and conclude that
\[
\m(P)=\m(\widetilde P),
\]
which completes the proof.
\end{proof}
\section{Applications of Theorem \ref{thm:main}} \label{sec:applications}

As remarked in the introduction, by taking $P(x,y,z)=x+1+(x-1)(y+z)$, and by considering its Mahler measure given in \eqref{eq:Condon},  we can deduce equations \eqref{eq:2}, \eqref{eq:4},  and \eqref{eq:5} by  setting $g(x)=x+2, x^2-2x+2,$ and $x^4+x+2$ with $k=2, 4,$ and $5$ respectively, with $\lambda=1$,  and after subtracting $\log 2$ from both sides of the evaluated identity in Theorem \ref{thm:main}. 

Theorem \ref{thm:main} can be applied to very general polynomials $P$ and even to rational functions. Indeed, to work with rational functions, we can apply Theorem \ref{thm:main} to the numerator and the denominator individually, and then put them back together to recover the rational function with the change of variables. 

Here we examine an interesting application involving  rational functions with arbitrarily many variables whose Mahler measure can be computed explicitly.  Considering the fact that calculating the exact Mahler measure of multivariable polynomials is generally hard, let alone those with arbitrarily many variables, this is an intriguing and rare result.

Let
\begin{align*}
R_m(x_1,\dots, x_m,z):=&z+\left(\frac{1-x_1}{1+x_1}\right)\cdots\left(\frac{1-x_m}{1+x_m}\right),\\
S_m(x_1,\dots,x_m,x,y,z):=&  (1+x)z +  \left(\frac{1-x_1}{1+x_1}\right)
\cdots \left( \frac{1-x_{m}}{1+x_{m}}\right)(1+y),\\
T_m(x_1,\dots,x_m,x,y):=&1 + \left( \frac{1 - x_1}{1+x_1}\right)\cdots  \left( \frac{1 - x_{m}}{1+x_{m}} \right)x + \left( 1 -\left(\frac{1 - x_1}{1+x_1} \right)\cdots  \left( \frac{1 - x_{m}}{1+x_{m}}\right)\right) y. 
\end{align*}

For $a_1, \dots a_m \in \C$, define
\[s_\ell(a_1,\dots, a_m)=\left\{
\begin{array}{ll}
1& \mbox{if }\ell=0,\\
\sum_{i_1<\cdots<i_\ell}a_{i_1}\cdots a_{i_\ell}& \mbox{if }0<\ell\leq m,\\
0& \mbox{if }m<\ell.
\end{array}
\right.\]
Recall that the Bernoulli numbers are given by 
\[\frac{x}{e^x-1}=\sum_{n=0}^\infty \frac{B_n x^n}{n!}.\]

The Mahler measures of the polynomials $R_m,S_m,T_m$ can be computed by the following formulas. 
\begin{thm}\label{thm:nvar} (\cite{nvar,LL}) We have the following identities.
For $n \geq 1$,
 \[\m\left(R_{2n} \right) =\sum_{h=1}^n  \frac{s_{n-h}(2^2, 4^2, \dots, (2n-2)^2)}{(2n-1)!} \left(\frac{2}{\pi}\right)^{2h}\mathcal{A}(h)
,\]
where
\[\mathcal{A}(h):=(2h)!\left(1-\frac{1}{2^{2h+1}}\right) \zeta(2h+1).\]

For $n \geq 0$,
\[ \m\left( R_{2n+1}\right)= \sum_{h=0}^n \frac{s_{n-h}(1^2,3^2, \dots, (2n-1)^2)}{(2n)!}  \left(\frac{2}{\pi}\right)^{2h+1} \mathcal{B}(h),\]
where
\[\mathcal{B}(h):=(2h+1)!L(\chi_{-4}, 2h+2).\]

For $n \geq 1$,
\[\m\left(S_{2n}\right)=\sum_{h=1}^n\frac{ s_{n-h}(2^2,4^2, \dots, (2n-2)^2)}{(2n-1)!}\left(\frac{2}{\pi}\right)^{2h+2}\mathcal{C}(h),\]
where
\[ \mathcal{C}(h):=\sum_{\ell=1}^{h}\binom{2h}{2\ell}
\frac{(-1)^{h-\ell} }{4h} B_{2(h-\ell)} \pi^{2h-2\ell} (2\ell+2)! \left(1-\frac{1}{2^{2\ell+3}}\right) \zeta(2\ell+3). \]

For $n \geq 0$,
\[  \m\left( S_{2n+1} \right)=\sum_{h=0}^{n} \frac{s_{n-h}(1^2, 3^2 \dots, (2n-1)^2)} {(2n)!} \left(\frac{2}{\pi}\right)^{2h+3}\mathcal{D}(h),\]
where
\[\mathcal{D}(h):=\sum_{\ell=0}^{h}\binom{2h+1}{2\ell+1}\frac{(-1)^{h-\ell}}{2(2h+1)}B_{2(h-\ell)}\pi^{2h-2\ell}(2\ell+3)!L(\chi_{-4},2\ell+4).\]

For $n\geq 1$,
\[ \m \left ( T_{2n} \right)= \frac{\log 2}{2} +\sum_{h=1}^n \frac{s_{n-h}(2^2, 4^2, \dots, (2n-2)^2)}{(2n-1)!} \left(\frac{2}{\pi}\right)^{2h}\mathcal{E}(h),\]
where
\begin{align*}
\mathcal{E}(h):=& \frac{(2h)!}{2}
\left(1-\frac{1}{2^{2h+1}}\right) \zeta(2h+1)+  \sum_{\ell=1}^{h}  (2^{2(h-\ell)-1}-1)\binom{2h}{2\ell}\frac{ (-1)^{h-\ell+1}}{2h} \\&\times B_{2(h-\ell)}
\pi^{2h-2\ell} (2\ell)!\left(1-\frac{1}{2^{2\ell+1}}\right)\zeta(2\ell+1).
\end{align*}

For $n\geq 0$,
\[ \m \left ( T_{2n+1}\right)= \frac{\log 2}{2} +\sum_{h=1}^{n}\frac{s_{n-h}(2^2,4^2, \dots,(2n-2)^2)}{(2n+1)!}\left(\frac{2}{\pi}\right)^{2h+2} \mathcal{F}(h),\]
where \begin{align*}\mathcal{F}(h):=& \frac{(2h+2)!}{2}
\left(1-\frac{1}{2^{2h+3}}\right)  \zeta(2h+3)+   \frac{\pi^2n^2}{2}(2h)!
\left(1-\frac{1}{2^{2h+1}}\right)  \zeta(2h+1) \\
& + n(2n+1)\sum_{\ell=1}^{h} (2^{2(h-\ell)-1}-1) \binom{2h}{2\ell}
\frac{(-1)^{h-\ell+1}}{4h} B_{2(h-\ell)} \pi^{2h+2-2\ell}(2\ell)!
\left(1-\frac{1}{2^{2\ell+1}}\right)\zeta(2\ell+1).
\end{align*}
\end{thm}

We can apply Theorem \ref{thm:main} to each variable $x_\ell$, replacing it with $f(x_\ell)/g(x_\ell)$ for any compatible $f(x_\ell)$, $g(x_\ell)$. This yields infinitely many more formulas satisfying the results of Theorem \ref{thm:nvar}. For example, $g(x_1)=x_1+2$, $k=2$ and $\lambda=1$ gives, for $m
\geq 1$
\begin{align*}
\m\left(z+\left(\frac{x_1^2-1}{x_1^2+x_1+1}\right)\left(\frac{1-x_2}{1+x_2}\right)\cdots\left(\frac{1-x_m}{1+x_m}\right)\right) =&~\m\left(R_m \right),\\
\m\left((1+x)z + \left(\frac{x_1^2-1}{x_1^2+x_1+1}\right) \left(\frac{1-x_2}{1+x_2}\right)
\cdots \left( \frac{1-x_{m}}{1+x_{m}}\right)(1+y) \right)=&~ \m\left(S_m\right),\\
\end{align*}
and 
\begin{align*}
\m\left(1 + \left(\frac{x_1^2-1}{x_1^2+x_1+1}\right)\right. & \left. \left( \frac{1 - x_2}{1+x_2}\right)\cdots  \left( \frac{1 - x_{m}}{1+x_{m}} \right)x \right. \\&\left. + \left( 1 -\left(\frac{x_1^2-1}{x_1^2+x_1+1}\right)\left(\frac{1 - x_2}{1+x_2} \right)\cdots  \left( \frac{1 - x_{m}}{1+x_{m}}\right)\right) y\right)=\m(T_m). 
\end{align*}
In particular, equation \eqref{eq:93zeta5} follows from considering the identity with $S_2$. We remark that in this case the formula for the Mahler measure of the rational function $\widetilde{S_2}$ can be directly written as the Mahler measure of the polynomial in \eqref{eq:93zeta5} by multiplying by the denominator, which has Mahler measure zero since it is the product of cyclotomic polynomials. 
Similarly, equations \eqref{eq:93zeta5a} and \eqref{eq:93zeta5b} are obtained by using $g(x)=x^2-2x+2,$ and $x^4+x+2$ with $k=4,$ and $5$ and $\lambda=1$ with $S_2$.  

This application is yet another example emphasising the versatility of Theorem \ref{thm:main} towards expanding existing results to generate new identities, that would otherwise be very difficult to establish.

\section{A discussion on the genus of the involved varieties} \label{sec:genus}
In this section we briefly discuss the relation of Mahler measure with the regulator in order to give additional context to the formulas that we have proven. More specifically, we explain how some of the new polynomials correspond to curves of higher genus, which motivates our interest in such formulas.

We start by following Deninger \cite{Deninger} as well as the treatment by Brunault and Zudilin \cite[Section 8.3]{BrunaultZudilin}. Let $F\in \C[y_1,\dots,y_{n+1}]$ be an irreducible polynomial in $(n+1)$ variables. As in Section \ref{sec:intermediate}, we write 
\[F=c_0+c_1y_{n+1}+\cdots+c_ky_{n+1}^k,
\]
where $c_i \in \C[y_1,\dots,y_{n}]$. Applying Jensen's formula as in the proof of Lemma \ref{lem:exist}, we can write 
\[\m(F)=\m(c_k) +\frac{1}{(2\pi i)^n} \int_D \log|y_{n+1}| \frac{\mathrm{d} y_1}{y_1}\wedge \cdots \wedge \frac{\mathrm{d}y_n}{y_n},\]
where 
\[D=\{(y_1,\dots, y_n)\, :\, |y_1|=\cdots=|y_n|=1, |y_{n+1}|>1, F(y_1,\dots,y_{n+1})=0\}\]
is known as the Deninger cycle attached to $F$. 

The differential form $\log|y_{n+1}| \frac{\mathrm{d}y_1}{y_1}\wedge \cdots \wedge \frac{\mathrm{d}y_n}{y_n},$
defined over the smooth part $Z_F^{\mathrm{reg}}$ of the zero locus of $F$ in $(\C^\times)^{n+1}$, is not necessarily closed, but there is a closed form $\eta(y_1,\dots, y_{n+1})$, defined on the de Rham cohomology of $Z_F^{\mathrm{reg}}$, that satisfies 
\[\eta(y_1,\dots, y_{n+1})|_D=(-1)^n \log|y_{n+1}| \frac{\mathrm{d}y_1}{y_1}\wedge \cdots \wedge \frac{\mathrm{d}y_n}{y_n}.\]
We have the following theorem.
\begin{thm}{\cite[Proposition 3.3]{Deninger}, \cite[Theorem 8.11]{BrunaultZudilin}}\label{thm:regulator} Let $F \in \C[y_1,\dots, y_{n+1}]$ be an irreducible polynomial such that $\overline{D}$ is a topological $n$-chain contained in $Z_F^{\mathrm{reg}}$. Then 
\[\m(F)=\m(c_k) +\frac{(-1)^n}{(2\pi i)^n} \int_{\overline{D}} \eta(y_1,\dots,y_{n+1}).\]
 \end{thm}
 The form $\eta$ represents the Beilinson's regulator evaluated in the Minor symbol $\{y_1,\dots, y_{n+1}\}$.  Beilinson's conjecture, that we will not describe here, together with  Theorem \ref{thm:regulator}  suggest that the Mahler measure $\m(F)$ should be related to an $L$-function associated to $Z_F^{\mathrm{reg}}$ when the boundary of the Deninger cycle is {\em trivial}. There are several technicalities that need to be taken into consideration (for example, the fact that $Z_F^{\mathrm{reg}}$ is not projective), that we will not discuss here. 
  
A special case arises when $\eta(y_1,\dots, y_{n+1})$ is an exact form and the boundary of the Deninger cycle $\partial \overline{D}$ is {\em non-trivial}. Writing $\eta(y_1,\dots, y_{n+1})=\mathrm{d}\, \omega(y_1,\dots, y_{n+1})$, Stokes's formula gives 
\[\m(F)=\m(c_k) +\frac{(-1)^n}{(2\pi i)^n} \int_{\partial \overline{D}} \omega(y_1,\dots,y_{n+1}),\]
where $\partial \overline{D}$ is the oriented boundary of the Deninger cycle. 
As proposed by Maillot \cite{Maillot} after an idea of Darboux \cite{Darboux},
 $\partial \overline{D}$ is contained in the algebraic subvariety 
 \[W_F : F(y_1,\dots, y_{n+1})=\overline{F}\left( \frac{1}{y_1},\dots,\frac{1}{y_{n+1}}\right)=0,\]
 and this suggests that $\m(F)$ is related to an $L$-function associated to $W_F$. 
 
 Under these conditions, the computation of the Mahler measure may proceed depending on the properties of $\omega(y_1,\dots,y_{n+1})$. A possibility is that $\omega$ is
 exact, and we can further apply Stokes's formula. 

For example, in three variables, when the $1$-form $\omega$ is exact, its primitive can be expressed in terms of a trilogarithm. This is the case of formulas \eqref{eq:Smyth3}, \eqref{eq:Condon}, or the polynomials $R_2$ and $T_1$ (the polynomial $S_0$ has the same Mahler measure as \eqref{eq:Smyth3}). The Mahler measure formulas corresponding to these polynomials result in the evaluation of a trilogarithm in the intersection of  $\{ |x|=|y| = |z|=1 \}$ with a curve $W_F$ of genus 0. For example, let us look at \eqref{eq:Condon}. Then \[W_F:(x+1)(xy^2-y^2+xy+y-x+1)=0,\] the union of a line and a genus 0 curve. 
What is interesting about applying Theorem \ref{thm:main} is that it results in evaluations of {\em higher genus curves} that still lead to {\em simple Mahler measure formulas}.  For example, let us look at \eqref{eq:2}. Then \[W_F: (x^2+x+1)(x^2y^2-y^2+x^2y+xy+y-x^2+1)=0,\] and the main factor is a genus 1 curve. Analogously, 
 if we consider the corresponding curves $W_F$ for formulas  \eqref{eq:4} and \eqref{eq:5}, we obtain genus $3$ and $4$ curves respectively. 

In other cases, such as in formula \eqref{eq:L21}, $\omega$ is not exact. This poses a challenge to the evaluation of the integral, which explains why formula \eqref{eq:L21} is only known numerically, but it has not been proven. In this case, the integration takes place in a curve $W_F$ of genus 1, which corresponds, precisely, to the elliptic curve $21a1$. Here Theorem \ref{thm:main} results in  evaluations of higher genus curves. For example, if we consider the corresponding curves $W_F$ for formulas \eqref{eq:L21-2},  \eqref{eq:L21-3}, and \eqref{eq:L21-4}, we obtain genus $3, 7$, and $9$ respectively. 

In conclusion, Theorem \ref{thm:main} leads to polynomials with substantially more complicated geometry than the original polynomials but with the same Mahler measure. This process serves to generate highly non-trivial identities.

\bibliographystyle{amsalpha}
\bibliography{Bibliography}

\providecommand{\bysame}{\leavevmode\hbox to3em{\hrulefill}\thinspace}
\providecommand{\MR}{\relax\ifhmode\unskip\space\fi MR }
\providecommand{\MRhref}[2]{%
  \href{http://www.ams.org/mathscinet-getitem?mr=#1}{#2}
}
\providecommand{\href}[2]{#2}
\begin{thebibliography}{BRVD03}

\bibitem[Boy81]{B1}
David~W. Boyd, \emph{Speculations concerning the range of {M}ahler's measure},
  Canad. Math. Bull. \textbf{24} (1981), no.~4, 453--469. \MR{644535}

\bibitem[Boy98]{Bo98}
\bysame, \emph{Mahler's measure and special values of {$L$}-functions},
  Experiment. Math. \textbf{7} (1998), no.~1, 37--82. \MR{1618282}

\bibitem[Boy06]{Boyd06}
David~W. Boyd, \emph{Conjectural explicit formulas for the {M}ahler measure of
  some three variable polynomials}, Personal communication, 2006.

\bibitem[BRV02]{Boyd-RV}
David~W. Boyd and Fernando Rodriguez-Villegas, \emph{Mahler's measure and the
  dilogarithm. {I}}, Canad. J. Math. \textbf{54} (2002), no.~3, 468--492.
  \MR{1900760}

\bibitem[BRVD03]{Boyd-RV2}
David~W. Boyd, Fernando Rodriguez-Villegas, and Nathan~M. Dunfield,
  \emph{Mahler's measure and the dilogarithm {II}}, arXiv:math/0308041 (2003).

\bibitem[BZ20]{BrunaultZudilin}
Fran\c{c}ois Brunault and Wadim Zudilin, \emph{Many variations of {M}ahler
  measures---a lasting symphony}, Australian Mathematical Society Lecture
  Series, vol.~28, Cambridge University Press, Cambridge, 2020. \MR{4382435}

\bibitem[Che95]{Chen}
Weiyu Chen, \emph{On the polynomials with all their zeros on the unit circle},
  J. Math. Anal. Appl. \textbf{190} (1995), no.~3, 714--724. \MR{1318593}

\bibitem[Coh22]{Cohn}
A.~Cohn, \emph{\"{U}ber die {A}nzahl der {W}urzeln einer algebraischen
  {G}leichung in einem {K}reise}, Math. Z. \textbf{14} (1922), no.~1, 110--148.
  \MR{1544543}

\bibitem[Con04]{Condon}
John~Donald Condon, \emph{Mahler measure evaluations in terms of
  polylogarithms}, ProQuest LLC, Ann Arbor, MI, 2004, Thesis (Ph.D.)--The
  University of Texas at Austin. \MR{2706324}

\bibitem[Dar75]{Darboux}
Gaston Darboux, \emph{M\'{e}moire sur les fonctions discontinues}, Ann. Sci.
  \'{E}cole Norm. Sup. (2) \textbf{4} (1875), 57--112. \MR{1508624}

\bibitem[Den97]{Deninger}
Christopher Deninger, \emph{Deligne periods of mixed motives, {$K$}-theory and
  the entropy of certain {${\bf Z}^n$}-actions}, J. Amer. Math. Soc.
  \textbf{10} (1997), no.~2, 259--281. \MR{1415320}

\bibitem[EW99]{EverestWard}
Graham Everest and Thomas Ward, \emph{Heights of polynomials and entropy in
  algebraic dynamics}, Universitext, Springer-Verlag London, Ltd., London,
  1999. \MR{1700272}

\bibitem[Lal06]{nvar}
Matilde~N. Lal\'in, \emph{Mahler measure of some {$n$}-variable polynomial
  families}, J. Number Theory \textbf{116} (2006), no.~1, 102--139.
  \MR{2197862}

\bibitem[Lal07]{L-Duke}
\bysame, \emph{An algebraic integration for {M}ahler measure}, Duke Math. J.
  \textbf{138} (2007), no.~3, 391--422. \MR{2322682}

\bibitem[Lal08]{L08}
\bysame, \emph{Mahler measures and computations with regulators}, J. Number
  Theory \textbf{128} (2008), no.~5, 1231--1271. \MR{2406490}

\bibitem[Lal15]{Lalin-Crelle}
\bysame, \emph{Mahler measure and elliptic curve {$L$}-functions at {$s=3$}},
  J. Reine Angew. Math. \textbf{709} (2015), 201--218. \MR{3430879}

\bibitem[Lax44]{Lax}
Peter~D. Lax, \emph{Proof of a conjecture of {P}. {E}rd\"{o}s on the derivative
  of a polynomial}, Bull. Amer. Math. Soc. \textbf{50} (1944), 509--513.
  \MR{10731}

\bibitem[Leh33]{Le}
D.~H. Lehmer, \emph{Factorization of certain cyclotomic functions}, Ann. of
  Math. (2) \textbf{34} (1933), no.~3, 461--479. \MR{1503118}

\bibitem[LL16]{LL}
Matilde~N. Lal\'in and Jean-S\'ebastien Lechasseur, \emph{Higher {M}ahler
  measure of an {$n$}-variable family}, Acta Arith. \textbf{174} (2016), no.~1,
  1--30. \MR{3517530}

\bibitem[LS13]{LalinSmyth}
M.~N. Lal\'{\i}n and C.~J. Smyth, \emph{Unimodularity of zeros of
  self-inversive polynomials}, Acta Math. Hungar. \textbf{138} (2013), no.~1-2,
  85--101. \MR{3015964}

\bibitem[LS15]{Lalin-Smyth-addendum}
\bysame, \emph{Addendum to: {U}nimodularity of zeros of self-inversive
  polynomials [ 3015964]}, Acta Math. Hungar. \textbf{147} (2015), no.~1,
  255--257. \MR{3391526}

\bibitem[Mah62]{Mah}
K.~Mahler, \emph{On some inequalities for polynomials in several variables}, J.
  London Math. Soc. \textbf{37} (1962), 341--344. \MR{0138593}

\bibitem[Mai00]{CM}
Vincent Maillot, \emph{G\'eom\'etrie d'{A}rakelov des vari\'et\'es toriques et
  fibr\'es en droites int\'egrables}, no.~80, 2000. \MR{1775582}

\bibitem[Mai03]{Maillot}
Vincent Maillot, \emph{{M}ahler measure in {A}rakelov geometry}, Workshop
  lecture at {\em ``The many aspects of Mahler's measure''\/}, Banff
  International Research Station, Banff, Canada, 2003.

\bibitem[RV99]{RV}
F.~Rodriguez-Villegas, \emph{Modular {M}ahler measures. {I}}, Topics in number
  theory ({U}niversity {P}ark, {PA}, 1997), Math. Appl., vol. 467, Kluwer Acad.
  Publ., Dordrecht, 1999, pp.~17--48. \MR{1691309}

\bibitem[RZ14]{RZ14}
Mathew Rogers and Wadim Zudilin, \emph{On the {M}ahler measure of
  {$1+X+1/X+Y+1/Y$}}, Int. Math. Res. Not. IMRN (2014), no.~9, 2305--2326.
  \MR{3207368}

\bibitem[Smy81]{S1}
C.~J. Smyth, \emph{On measures of polynomials in several variables}, Bull.
  Austral. Math. Soc. \textbf{23} (1981), no.~1, 49--63. \MR{615132}

\end{thebibliography}
\end{document}